\documentclass[conference]{IEEEtran}

\usepackage[left=0.8in,right=0.75in,top=1in,bottom=0.8in,%
            footskip=.25in]{geometry}
\newcommand*{\email}[1]{\texttt{#1}}

\usepackage{mathtools,amssymb,graphicx,times,comment}
\usepackage{amsmath,amsthm}
\usepackage{bbm}
  {
      
  }
  \usepackage{breqn}

 \usepackage{amsmath}
  \usepackage{tcolorbox}
\usepackage[hidelinks]{hyperref}
\definecolor{darkred}{RGB}{150,0,0}
\definecolor{darkgreen}{RGB}{0,150,0}
\definecolor{darkblue}{RGB}{0,0,150}
\hypersetup{colorlinks=true, linkcolor=red, citecolor=darkgreen, urlcolor=darkblue}

\newcommand{\BSC}{{\rm{BSC}}}
\newcommand{\BSCe}{{\rm{BSC}}_{\eps}}
\newcommand{\mathleft}{\@fleqntrue\@mathmargin0pt}
\newcommand{\mathcenter}{\@fleqnfalse}
\newcommand{\corr}[2]{{\rm{corr}}\left(\,{#1}\,;\,{#2}\,\right)}

\newcommand{\env}[3]{\mathcal{M}_{{#1}}\left({#2};{#3}\right)}
\newcommand{\prox}[3]{\mathrm{prox}_{{#1}}\left({#2};{#3}\right)}

\newcommand{\proxl}[2]{\mathrm{prox}_{\ell}\left({#1};{#2}\right)}

\newcommand{\elldd}{\ell^{''}}
\newcommand{\elld}{\ell^{'}}

\newcommand{\envdx}[3]{\mathcal{M}^{'}_{{#1},1}\left({#2};{#3}\right)}
\newcommand{\envdla}[3]{\mathcal{M}^{'}_{{#1},2}\left({#2};{#3}\right)}
\newcommand{\envddx}[3]{\mathcal{M}^{''}_{{#1},1}\left({#2};{#3}\right)}

\newcommand{\ourx}{\al G + \mu S Y}

\newcommand{\R}{\mathbb{R}}

\newcommand{\al}{\alpha}

\newcommand{\xh}{\widehat{\x}}
\newcommand{\yo}{\overline{y}}

\providecommand{\abs}[1]{\lvert#1\rvert}
\providecommand{\norm}[1]{\lVert#1\rVert}

\DeclarePairedDelimiterX{\inp}[2]{\langle}{\rangle}{#1, #2}

\newcommand{\ksi}{\xi}

\usepackage{dsfont}

\newcommand{\simiid}{\stackrel{\text{iid}}{\sim}}

\newcommand{\Pro}{\mathbb{P}}

\newcommand{\soft}[2]{{\Hc}\left({#1};{#2}\right)}

\theoremstyle{theorem}
\newtheorem{propo}{Proposition}[section]
\newtheorem{thm}{Theorem}[section]
\newtheorem{cor}{Corollary}[section]
\newtheorem{ass}{Assumption}
\theoremstyle{remark}
\newtheorem{remark}{Remark}

\theoremstyle{definition}

\newcommand{\eps}{\varepsilon}

\newcommand{\sign}{\mathrm{sign}}

\newcommand{\Exp}{\mathbb{E}}               
\newcommand{\E}{\mathbb{E}}                    
\newcommand{\la}{{\lambda}}                     

\newcommand{\sig}{\sigma}

\newcommand{\nn}{\notag}

\newcommand{\G}{\mathbf{G}}

\newcommand{\A}{\mathbf{A}}

\newcommand{\x}{\mathbf{x}}
\newcommand{\ub}{\mathbf{u}}

\newcommand{\w}{\mathbf{w}}

\newcommand{\g}{\mathbf{g}}
\newcommand{\vb}{\mathbf{v}}

\newcommand{\y}{\mathbf{y}}
\newcommand{\s}{\mathbf{s}}

\newcommand{\ab}{\mathbf{a}}

\newcommand{\h}{\mathbf{h}}

\newcommand{\betab}{\boldsymbol{\beta}}

\newcommand{\Fc}{\mathcal{F}}
\newcommand{\Mc}{\mathcal{M}}

\newcommand{\Nn}{\mathcal{N}}

\newcommand{\Hc}{\mathcal{H}}

\newcommand{\Ic}{\mathcal{I}}

\newcommand{\beq}{\begin{equation}}
\newcommand{\eeq}{\end{equation}}
\newcommand{\bea}{\begin{align}}
\newcommand{\eea}{\end{align}}

\newcommand{\vp}{\vspace{4pt}}

\usepackage{amsmath}
\def\bea#1\eea{\begin{align}#1\end{align}}

\title{Sharp Guarantees for Solving Random Equations\\ with One-Bit Information}
\author{
Hossein Taheri, Ramtin Pedarsani, and Christos Thrampoulidis
\thanks{
Electrical and Computer Engineering Department, University of California, Santa Barbara, Santa Barbara, CA 93106, USA. Emails: 
\email{\{hossein, ramtin, cthrampo\}@ucsb.edu}}
}

\IEEEoverridecommandlockouts
\begin{document}
\maketitle
\begin{abstract}
%
We study the performance of a wide class of convex optimization-based estimators for recovering a signal from corrupted one-bit measurements in high-dimensions. Our general result  predicts sharply the performance of such estimators in the linear asymptotic regime when the measurement vectors have entries IID Gaussian. This includes, as a special case, the previously studied least-squares estimator and various novel results for other popular estimators such as least-absolute deviations, hinge-loss and logistic-loss. Importantly, we exploit the fact that our analysis holds for generic convex loss functions to prove a bound on the best achievable performance across the entire class of estimators. 
Numerical simulations corroborate our theoretical findings and suggest they are accurate even for relatively small problem dimensions. 
\end{abstract}

\section{Introduction}

%
%
%

\subsection{Motivation}
Classical statistical signal-processing theory studies estimation problems in which the number of unknown parameters $n$ is small compared to the number of observations $m$. In contrast, modern inference problems are typically \emph{high-dimensional}, that is $n$ can be of the same order as $m$. Examples are abundant in a wide range of signal-processing applications such as medical imaging, wireless communications, recommendation systems and so on. Classical tools and theories are not applicable in these modern inference problems. As such, over the last two decades or so, the study of high-dimensional estimation problems has received significant attention. 
Despite the remarkable progress in many directions, several important questions remain to be explored. 

This paper studies the fundamental problem of recovering an unknown signal from (possibly corrupted) one-bit measurements in high-dimensions. We focus on a rather rich class of convex optimization-based estimators that includes, for example, least-squares (LS), least-absolute deviations (LAD), logistic regression and hinge-loss as special cases. For such estimators and Gaussian measurement vectors, we compute their asymptotic performance in the high-dimensional linear regime in which $m,n\rightarrow+\infty$ and $m/n\rightarrow\delta\in(1,+\infty)$. Importantly, our results are \emph{sharp}. In contrast to existing related results which are order-wise (i.e., they involve unknown or loose constants) this allows us to accurately compare the relative performance of different methods (e.g., LS vs LAD).
%
  It is worth mentioning that while our predictions are asymptotic, our numerical illustrations suggest that they are valid for dimensions $m$ and $n$ that are as small as a few hundreds. 


\subsection{Contributions}
 Our goal is to recover $\x_0\in\R^n$ from measurements $y_i=\sign(\ab_i^T\x_0),~i=1,\ldots,m$, where $\ab_i\in\R^n$ have entries iid Gaussian. The results account for possible corruptions by allowing each measurement $y_i$ to be sign-flipped with constant probability $\eps\in[0,1/2]$ (see Section \ref{sec:form} for details). We study the asymptotic performance of estimators $\hat\x_\ell$ that are solutions to the following optimization problem for some convex loss function $\ell(\cdot)$.
\bea\label{eq:intro_opt}
\xh_\ell := \arg\min_\x \sum_{i=1}^{m} \ell(y_i\ab_i^T\x).
\eea 
 When $m,n\rightarrow+\infty$ and $m/n\rightarrow\delta>\delta_\eps^\star$, we show that the correlation of $\xh_\ell$ to the true vector $\x_0$ is sharply predicted by $\sqrt{\frac{1}{1+({\al}/{\mu})^2}}$ where the parameters $\alpha$ and $\mu$ are the solutions to a system of three non-linear equations in three unknowns. We find that the system of equations (and thus, the value of $\alpha/\mu$) depends on the loss function $\ell(\cdot)$ through its Moreau envelope function. We prove that $\delta_\eps^\star>1$ is necessary for the equations to have a bounded solution, but, in general, the value of the threshold $\delta_\eps^\star$ depends both on the noise level $\eps$ and on the loss function. For the general case where $\eps \in[0,1/2]$, we propose a method to find the upper bound on correlation which corresponds to the maximum value of correlation that can be achieved for any convex loss function
\par Despite the fact that it is not the main focus of this paper, we remark that our results hold for the general case where measurements $y_i$ are determined according to an arbitrary function $f:\mathbb{R}\rightarrow[0,1]$ (see \eqref{eq:gen_label}). Moreover as our analysis in Appendix \ref{sec:mainproof} shows, the system of equations introduced in Section \ref{sec:SOE} can be extended to address regularized loss functions i.e., 
\bea
\sum_{i=1}^{m} \ell(y_i\ab_i^T\x) + r\left\lVert\x\right\rVert_2^2.\nonumber
\eea
\par We specialize our general result to specific loss functions such as LS, LAD and hinge-loss. This allows us to numerically compare the performance of these popular estimators by simply evaluating the corresponding theoretical predictions. Our numerical illustrations corroborate our theoretical predictions. For LS, our equations can be solved in closed form and recover the result of \cite{NIPS} (see Section \ref{sec:prior}). For the hinge-loss, we show that $\delta_\eps^\star$ is a decreasing function of $\eps$ that approaches $+\infty$ in the noiseless case and $2$ when $\eps=1/2$. 
 We believe that our work opens the possibility for addressing several important open questions, such
 as finding the optimal choice of the loss function in \eqref{eq:intro_opt} and the value of $\delta_\eps^\star$ for general loss functions.

\subsection{Prior work}\label{sec:prior}
As mentioned, over the past two decades there has been a very long list of works that derive statistical guarantees for high-dimensional estimation problems. In particular, many of these are concerned with convex optimization-based inference methods. Our work is most closely related to the following two lines of research.

\vp
\noindent\emph{(a)~Sharp asymptotic predictions for noisy linear measurements.} Most of the results in the literature of high-dimensional statistics are order-wise in nature. Sharp asymptotic predictions have only recently appeared in the literature for the case of noisy linear measurements with Gaussian measurement vectors. There are by now three different approaches that have been used (to different extent each) towards asymptotic analysis of convex regularized estimators: (a) the one that is based on the approximate message passing (AMP) algorithm and its state-evolution analysis; \cite{AMP,donoho2011noise,bayati2011dynamics,montanariLasso,donoho2016high} (b) the one that is based on Gaussian process (GP) inequalities, specifically the convex Gaussian min-max Theorem (CGMT); \cite{Sto,Cha,StoLASSO,OTH13,COLT,Master} (c) and the ``leave-one-out" approach \cite{karoui2013asymptotic,karoui15}. The three approaches are quite different to each other and each comes with its unique distinguishing features and disadvantages. A detailed comparison is beyond our scope. In this paper, we follow the GP approach and build upon the CGMT. 
Since concerned with linear measurements, these previous works consider estimators that solve minimization problems of the form
\bea\label{eq:intro_opt2}
\xh := \arg\min_\x \sum_{i=1}^{m} \widetilde{\ell}(y_i-\ab_i^T\x) + r R(\x)
\eea
Specifically, the loss function $\widetilde\ell(\cdot)$ penalizes the residual. In this paper, we extend the applicability of the CGMT to optimization problems in the form of \eqref{eq:intro_opt}. For our case of signed measurements, \eqref{eq:intro_opt} is more general than \eqref{eq:intro_opt2}. To see this, note that for $y_i\in\pm 1$ and popular symmetric loss functions $\widetilde{\ell}(t)=\widetilde{\ell}(-t)$ (e.g., LS, LAD), \eqref{eq:intro_opt} results in \eqref{eq:intro_opt2} by choosing $\ell(t)=\widetilde\ell(t-1)$ in the former. Moreover, \eqref{eq:intro_opt} includes several other popular loss functions such as the logistic loss and the hinge-loss which cannot be expressed by \eqref{eq:intro_opt2}.

\vp
\noindent\emph{(b)~One-bit compressed sensing.} Our works naturally relates to the literature of one-bit compressed sensing (CS) \cite{boufounos20081}. The vast majority of performance guarantees for one-bit CS are order-wise in nature, e.g., \cite{jacques2013robust,plan2013one,plan2012robust,PV15}. To the best of our knowledge, the only existing sharp results are presented in \cite{NIPS} for Gaussian measurement vectors. Specifically, the paper \cite{NIPS} derives the asymptotic performance of regularized LS for generalized nonlinear measurements, which include signed measurements as a special case. Our work can be seen as a direct extension of \cite{NIPS} to loss functions beyond least-squares, such as the hinge-loss. In fact, the result of \cite{NIPS} for our setting is a direct corollary of our main theorem (see Section \ref{sec:LS}). As in \cite{NIPS}, our proof technique is based on the CGMT.  

\vp
There are few works that consider general convex loss functions for estimating a signal from noisy measurements in high dimensions. In \cite{genzel2016nonas}, the general estimator $\ell(\langle\ab_i,\x\rangle,y_i)$ for estimating a structured signal in the non-asymptotic case has been studied. However it is assumed that the loss function satisfies some conditions including restricted strong convexity, continuously differentiability in the first argument and derivative of loss function being Lipschitz-continuous with respect to the second argument. The author furthermore derives some sufficient conditions for the loss function ensuring the restricted strong convexity condition. 
Our result in Theorem \ref{sec:lem} for achieving the best performance across all loss functions is comparable to \cite[Theorem 1]{bean2013optimal} in which they have also proposed a method for deriving optimal loss function and measuring its performance. However their results hold for measurements of the form $y_i=\ab_i^T\x_0+ \eps_i$, where $\{\eps_i\}_{i=1}^{m}$ are random errors independent of  $\ab_i$'s. 

\vp
Finally, our paper is closely related to \cite{candes2018phase,sur2019modern}, in which the authors study the high-dimensional performance of maximum-likelihood (ML) estimation for the logistic model. The ML estimator is a special case of \eqref{eq:intro_opt} but their measurement model differs from the one considered in this paper. Also, their analysis is based on the AMP. While this paper was being prepared, we became aware of \cite{salehi2019impact}, in which the authors extend the results of \cite{sur2019modern} to regularized ML by using the CGMT. However we present results for general loss functions and a different measurement model.

\subsection{Organization and notation}
The rest of the paper is organized as follows. Section \ref{sec:form} formally introduces the problem that this paper is concerned with. We present our main results Theorems \ref{thm:main} and \ref{sec:lem} in Section \ref{sec:gen}, where we also discuss some of their implications. In Section \ref{sec:cases}, we specialize the general result of Theorem \ref{thm:main} to the LS, LAD and hinge-loss estimators. We also present numerical simulations to validate our theoretical predictions. We conclude in Section \ref{sec:conc} with several possible directions for future research. A proof sketch of Theorem \ref{thm:main} is provided in Appendix \ref{sec:proof}.

The symbols $\Pro\left(\cdot\right)$ and $\Exp\left[\cdot\right]$ denote the probability of an event and the expectation of a random variable, respectively. We use boldface notation for vectors. $\|\vb\|_2$ denotes the Euclidean norm of a vector $\vb$. We write $i\in[m]$ for $i=1,2,\ldots,m$. 
 When writing $x_* = \arg\min_x f(x),$ we let the  operator $\arg\min$ return any one of the possible minimizers of $f$. For all $x\in\mathbb{R}$, Gaussian $Q$-function at $x$ is defined as $Q(x)= \int_{s=x}^{\infty}\frac{1}{\sqrt{2\pi}}e^{s^2/2} ds.$
For a random variable $H$ with density $p_H(h)$ that has a derivative $p_H^{'}(h)$ for all $h\in\mathbb{R}$, denote its score function $\ksi_h:=\frac{\partial}{\partial h}{\log p(h)}=\frac{p_H^{'}(h)}{p_H(h)}$. Fisher information of $H$ (e.g., \cite[Sec.~2]{barron1984monotonic}) is defined as
$\Ic(H):=\Exp[\,(\ksi_H)^2\,].$

\section{Problem statement}\label{sec:form}

The goal is to recover the unknown vector $\x_0\in\R^n$ from $m$ noisy signed measurements $y_i$. Let $\yo_i,~i\in[m]$ denote the \emph{noiseless} signed measurements $ \yo_i= \sign(\ab_i^T\x_0),~i\in[m]$, where $\ab_i\in\R^n$ are the measurement vectors. We assume the following noise model in which each measurement is corrupted (i.e., sign flipped) with some constant probability $\eps\in[0,1/2]$:
\bea\label{eq:gen_model}
y_i = \BSCe(\yo_i) := \begin{cases} \sign(\ab_i^T\x_0) &, \text{w.p.}~~1-\eps, \\  -\sign(\ab_i^T\x_0) &, \text{w.p.}~~\eps.  \end{cases}
\eea
We remark that all our results remain valid in the case of (potentially) adversarial noise in which $\eps\, m$ number of noiseless measurements $\yo_i$ are flipped. Nevertheless, for the rest of the paper, we focus on the measurement model in \eqref{eq:gen_model}.
This paper studies the recovery performance of estimates $\xh_{\ell}$ of $\x_0$ that are obtained by solving the following convex optimization problem:
\bea\label{eq:gen_opt}
\xh_{\ell} \in \arg\min_{\x} \frac{1}{m} \sum_{i=1}^{m} \ell(y_i \ab_i^T \x),
\eea
Here, $\ell:\R\rightarrow\R$ is a convex loss function and the subscript $\ell$ on the estimate $\xh_\ell$ emphasizes its dependence on the choice of the loss function. Different choices for $\ell(.)$ lead to popular specific estimators. For example, these include the following:
\begin{itemize}
\item Least-squares (LS): $\ell(t)=\frac{1}{2}(t-1)^2$,
\item Least-absolute deviations (LAD): $\ell(t)=|t-1|$,
\item Logistic maximum-likelihood: $\ell(t)=\log(1+e^{-t})$,
\item Ada-boost: $\ell(t)=e^{-t}$,
\item Hinge-loss: $\ell(t)=\max\{1-t\,,\,0\}$.
\end{itemize} 

Since we only observe sign-information, any information about the magnitude $\|\x_0\|_2$ of the signal $\x_0$ is lost. Thus, we can only hope to obtain an accurate estimate of the direction of $\x_0$.
We measure performance of the estimate $\xh_{\ell}$ by its (absolute) correlation value to $\x_0$, i.e.,
\bea\label{eq:corr}
\corr{\xh_\ell}{\x_0}:=\frac{|\,{\inp{\xh_\ell}{\x_0}}\,|}{\|\xh_\ell\|_2 \|\x_0\|_2} \in [0,1].
\eea
Obviously, we seek estimates that maximize correlation. \\
Although signed measurements obtained according to \eqref{eq:gen_model} are the main focus of this paper, we remark that all results are valid for the general case where measurements are determined according to :
\bea \label{eq:gen_label}
y_i = \begin{cases} 1 &, \text{w.p.}~~ f(\ab_i^T\x_0), \\ -1 &, \text{w.p.}~~1-f(\ab_i^T\x_0).  \end{cases}
\eea
where $f : \mathbb{R}\rightarrow[0,1]$. It is straightforward to see that choosing $f(t) = \frac{1}{2}+\frac{1-2\eps}{2} \sign(t)$ will give \eqref{eq:gen_model}.
Moreover as the analysis in Appendix \ref{sec:mainproof} suggests, we can extend our results to the broader class of optimization problems where there is an additional term for the penalty on the magnitude of $\x$, i.e.,
 \bea \label{eq:opt_reg}
 \xh_\ell = \arg\min_\x \sum_{i=1}^{m} \ell(y_i\ab_i^T\x) + r\left\lVert\x\right\rVert_2^2,
 \eea
where $r \in \{ \mathbb{R}^+\cup{0}\}$. However, the main body of this paper focuses on the case  $r=0$.\\
Our main result characterizes the asymptotic value of $\corr{\xh_{\ell}}{\x_0}$ in the linear high-dimensional regime in which the problem dimensions $m$ and $n$ grow proportionally to infinity with ${m}/{n}\rightarrow\delta\in(1,\infty).$ All our results are valid under the assumption that the measurement vectors have entries IID Gaussian. 

\begin{ass}[Gaussian measurement vectors]\label{ass:Gaussian}
 The vectors $\ab_i,~i\in[m]$ have entries IID standard normal $\Nn(0,1)$.
\end{ass}

We make no further assumptions on the distribution of the true vector $\x_0$.






%

%
\section{Main results}\label{sec:gen}

\subsection{Moreau Envelopes}

Before presenting our main result, we need a few definitions. We write
$$\env{\ell}{x}{\la}:=\min_{v}\frac{1}{2\la}(x-v)^2 + \ell(v),$$
for the \emph{Moreau envelope function} of the loss $\ell:\R\rightarrow\R$ at $x$ with parameter $\la>0$. Note that the objective function in the minimization above is strongly convex. Thus, for all values of $x$ and $\la$, there exists a unique minimizer which we denote by $\prox{\ell}{x}{\la}$. This is known as the \emph{proximal operator} of $\ell$ at $x$ with parameter $\la$. One of the important and useful properties of the Moreau envelope function is that it is continuously differentiable with respect to both $x$ and $\la$ \cite{rockafellar2009variational}. We denote these derivatives as follows
\bea\nn
\envdx{\ell}{x}{\la}&:=\frac{\partial{\env{\ell}{x}{\la}}}{\partial x},\\
\envdla{\ell}{x}{\la}&:=\frac{\partial{\env{\ell}{x}{\la}}}{\partial \la}.\nn
\eea
The following is a well-known result that is useful for our purposes.
\begin{propo}[Derivatives of $\Mc_{\ell}$~\cite{rockafellar2009variational}]
\label{propo:der}
For a function $\ell:\R\rightarrow\R$, and all $x\in\R$ and $\lambda>0$, the following properties are true:
\bea
\envdx{\ell}{x}{\la} &= \frac{1}{\la}{(x-\prox{\ell}{x}{\la})}, \nn \\
\envdla{\ell}{x}{\la} &= -\frac{1}{2\la^2}{(x-\prox{\ell}{x}{\la})^2}.\nn
\eea
If in addition $\ell$ is differentiable and $\ell^{'}$ denotes its derivative, then
\bea
\envdx{\ell}{x}{\la}&= \elld(\proxl{x}{\la}),\nn\\
\envdla{\ell}{x}{\la}&= -\frac{1}{2}(\elld(\proxl{x}{\la})^2.\nn
\eea

\end{propo}

\subsection{A system of equations} \label{sec:SOE}

It turns out, that the asymptotic performance of \eqref{eq:gen_opt} depends on the loss function $\ell$ via its Moreau envelope. Specifically, define random variables $G,S$ and $Y$ as follows (recall the definition of $\BSCe$ in \eqref{eq:gen_model})
\bea\label{eq:GSY}
G,S\simiid\Nn(0,1) \quad\text{and}\quad Y=\BSCe(\sign(S)),
\eea
and consider the following system of non-linear equations in three unknowns $(\mu,\alpha>0,\la)$:
\begin{subequations}\label{eq:eq_main}
\bea
 \Exp\left[Y\, S \cdot\envdx{\ell}{\ourx}{\la}  \right]&=0 , \label{eq:mu_main}\\
 {\la^2}\,{\delta}\,\Exp\left[\,\left(\envdx{\ell}{\ourx}{\la}\right)^2\,\right]&=\alpha^2 ,
\label{eq:alpha_main}\\
\lambda\,\delta\,\E\left[ G\cdot \envdx{\ell}{\ourx}{\la}  \right]&=\alpha.
\label{eq:lambda_main}
\eea
\end{subequations}
The expectations above are with respect to the randomness of the random variables $G$, $S$ and $Y$. As we show shortly, the solution to these equations is tightly connected to the asymptotic behavior of the optimization in \eqref{eq:gen_opt}. 

%

We remark that the equations are well defined even if the loss function $\ell$ is not differentiable. If $\ell$ is differentiable then, using Proposition \ref{propo:der} the Equations \eqref{eq:eq_main}
 can be equivalently written as follows:
\begin{subequations}\label{eq:eq_main2}
\bea
 \Exp\left[Y\, S \cdot\elld\left( \proxl{\ourx}{\la} \right)  \right]\label{eq:mu_main2}&=0,\\
{\la^2}\,{\delta}\,\Exp\left[\left(\elld\left( \proxl{\ourx}{\la} \right)\right)^2\right]&=\alpha^2,  \label{eq:alpha_main2}\\
\lambda\,\delta\,\E\left[ G\cdot \elld\left(\proxl{\ourx}{\la}\right)  \right]&=\alpha. \label{eq:lambda_main2} 
\eea
\end{subequations}
Finally, if $\ell$ is two times differentiable then applying integration by parts  in Equation \eqref{eq:lambda_main2} results in the following reformulation of  \eqref{eq:lambda_main}:
\bea\label{eq:lambda_main3}
1 &= \lambda\,\delta\, \Exp\left[\frac{\elldd\left( \proxl{\ourx}{\la} \right)}{1+\la\, \elldd\left( \proxl{\ourx}{\la} \right)}\right].
\eea
Note that the system of equations in \eqref{eq:eq_main} and \eqref{eq:eq_main2} hold when there is no regularization i.e., $r=0$. For the general case, the system of equations can be extended to \eqref{eq:reg_main}.

\subsection{Asymptotic prediction}
We are now ready to state the main result of this paper.

\begin{thm}[General loss function]\label{thm:main}
 Let Assumption \ref{ass:Gaussian} hold and fix some $\eps\in[0,1/2]$ in \eqref{eq:gen_model}. Assume $\delta>1$ such that the set of minimizers in \eqref{eq:gen_opt} is bounded and the system of Equations \eqref{eq:eq_main} has a unique solution $(\mu,\al,\la)$, such that $\mu\neq 0$.  Let $\xh_\ell$ be as in \eqref{eq:gen_opt}. Then, in the limit of $m,n\rightarrow+\infty$, $m/n\rightarrow\delta$, it holds with probability one that

\bea\label{eq:corr_thm}
\lim_{n\rightarrow \infty} \corr{\xh_\ell}{\x_0} = \sqrt{\frac{1}{1+({\al}/{\mu})^2}}.
\eea

\noindent Moreover, 
\bea\label{eq:norm_thm}
\lim_{n\rightarrow \infty} \Big\|\xh_\ell-\mu\cdot\frac{\x_0}{\|\x_0\|_2}\Big\|_2^2 = \al^2.
\eea
\end{thm}

%

Theorem \ref{thm:main} holds for general loss functions. In Section \ref{sec:cases} we specialize the result to specific popular choices. We also present numerical simulations that confirm the validity of the predictions of Theorem \ref{thm:main} (see Figures \ref{fig:fig2}--\ref{fig:fig4}). Before that, in the following remarks we present a few notes on the conditions, interpretation and implications of the theorem. A proof outline is included in Appendix \ref{sec:proof}.

\begin{remark}[The role of $\mu$ and $\alpha$]
The theorem evaluates the asymptotic performance of the estimator $\xh_\ell$ for a convex loss function $\ell$ in \eqref{eq:gen_opt}. According to \eqref{eq:corr_thm}, the prediction for the limiting behavior of the correlation value is given in terms of $\sigma_\ell:=\alpha\big/\mu$, where $\mu$ and $\alpha$ are unique solutions of \eqref{eq:eq_main}. \emph{The smaller the value of $\sigma_\ell$ is, the larger becomes the correlation value.} Thus, the correlation value is fully determined by the ratio of the parameters $\alpha$ and $\mu$. Their individual role is clarified in \eqref{eq:norm_thm}. Specifically, according to \eqref{eq:norm_thm}, $\xh_{\ell}$ is a biased estimate of the true $\x_0$ and $\mu$ represents exactly that bias term. In other words, solving \eqref{eq:gen_opt} returns an estimator that is close to a $\mu$--scaled version of $\x_0$. When $\x_0$ and $\xh_{\ell}$ are scaled appropriately, then the L2 squared norm of their difference converges to $\alpha^2$.
\end{remark}
\begin{remark}[On the existence of a solution to \eqref{eq:gen_opt}] While $\delta>1$ is a necessary condition for the equations in \eqref{eq:eq_main} to have a solution, it is \emph{not} sufficient in general. This depends on the specific choice of the loss function. For example, in Section \ref{sec:LS}, we show that for the squared loss $\ell(t)=(t-1)^2$, the equations have a unique solution iff $\delta>1$. On the other hand, for  logistic regression and  hinge-loss, it is argued in Remark \ref{rem:threshold} that there exists a threshold value $\delta^\star_\eps:=\delta^\star(\eps)>2$ such that the set of minimizers in \eqref{eq:gen_opt} is unbounded if $\delta>\delta_{\eps}$. Hence, the theorem does not hold for $\delta<\delta^\star_\eps$.
 We conjecture that for these choices of loss, the equations \eqref{eq:eq_main} are solvable iff  $\delta>\delta_{\eps}$. Justifying this conjecture is an interesting direction for future work. More generally, we leave the study of sufficient and necessary conditions under which the equations \eqref{eq:eq_main} admit a unique solution to future work. 
\end{remark}

\begin{remark}[Bounded minimizers]\label{rem:threshold}
 Theorem \ref{thm:main} only holds in regimes for which the set of minimizers of \eqref{eq:gen_opt} is bounded. As we show here, this is $\emph{not}$ always the case. Specifically, consider non-negative loss functions $\ell(t)\geq 0$ with the property $\lim_{t\rightarrow+\infty} \ell(t)=0$. For example, the hinge-loss, Ada-boost and logistic loss all satisfy this property. Now, we show that for such loss functions the set of minimizers is unbounded if $\delta<\delta^\star_\eps$ for some appropriate $\delta^\star_\eps>2$. First, note that the set of minimizers is unbounded if the following condition holds:
\bea\label{eq:sep}
\exists~\x_s\in\R^p \quad\text{such that}\quad y_i\ab_i^T\x_s \geq 1, \quad \forall~i\in[m].
\eea
Indeed, if \eqref{eq:sep} holds then $\x=c\cdot\x_s$ with $c\rightarrow+\infty$, attains zero cost in \eqref{eq:gen_opt}; thus, it is optimal and the set of minimizers is unbounded. To proceed, we rely on a recent result by Candes and Sur \cite{candes2018phase} who prove that \eqref{eq:sep} holds iff \footnote{To be precise, \cite{candes2018phase} prove the statement for measurements $y_i,~i\in[m]$ that follow a logistic model. Close inspection of their proof shows that this requirement can be relaxed by appropriately defining the random variable $Y$ in \eqref{eq:threshold}.} \bea\label{eq:threshold}
\delta \leq \delta^\star_{\eps}:= \left(\min_{c\in\R}\E\left[\left(G+c\,S\,Y\right)_{-}^2\right]\right)^{-1},
\eea
where $G,S$ and $Y$ are random variables as in \eqref{eq:GSY} and $(t)_{-}:=\min\{0,t\}$. It can be checked analytically that $\delta^\star(\eps)$ is a decreasing function of $\eps$ with $\delta^\star(0^+)=+\infty$ and $\delta^{\star}(1/2)=2$. In Figure \ref{fig:thresholdfig}, we have numerically evaluated the threshold value $\delta^\star_\eps$ as a function of the corruption level $\eps$. For $\delta<\delta^\star_\eps$, the set of minimizers of the \eqref{eq:gen_opt} with logistic or hinge loss is unbounded. \\
\end{remark}
\begin{figure}
  \centering
    \includegraphics[width=8.2cm,height=7cm]{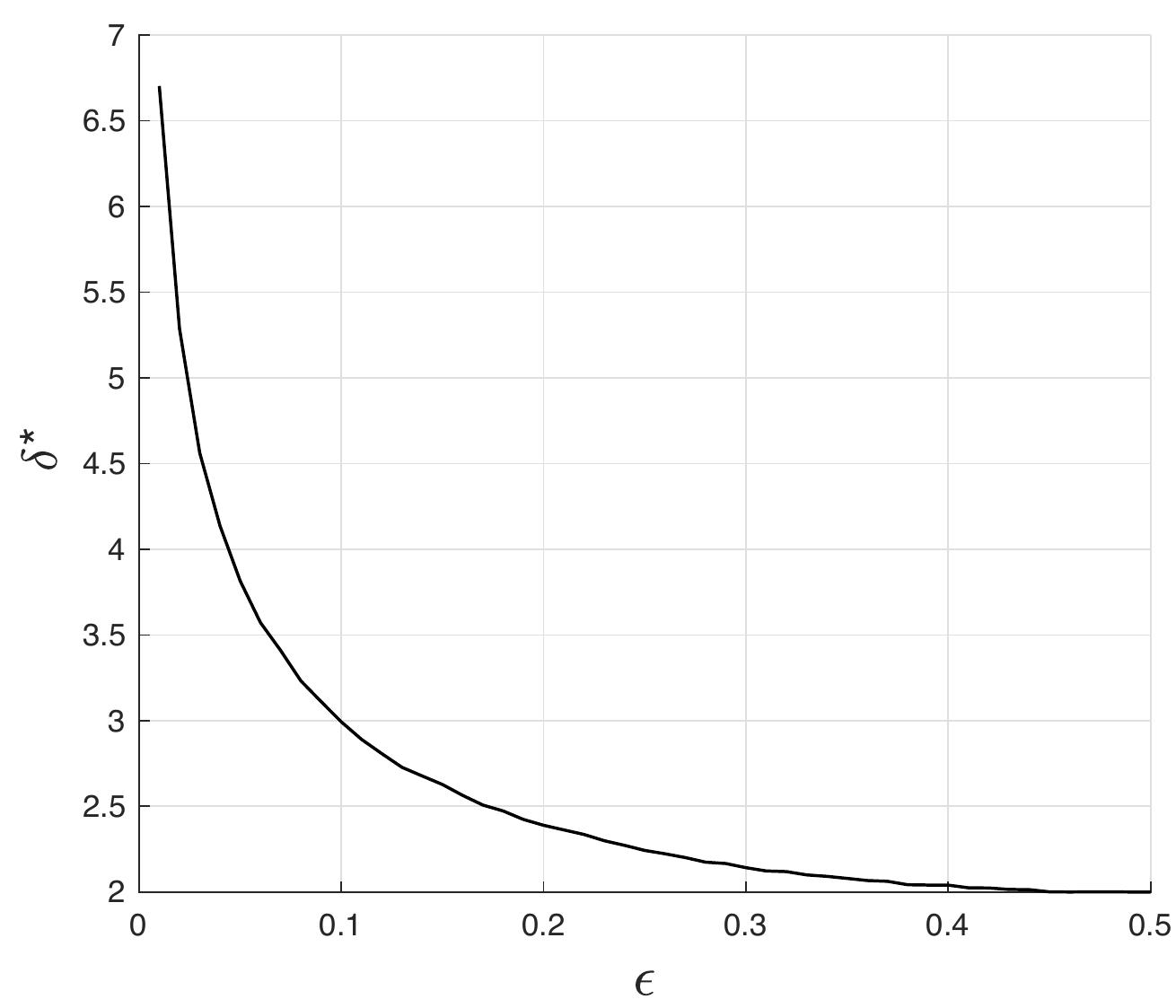}
      \caption{The value of the threshold $\delta^\star_\eps$ in \eqref{eq:threshold} as a function of probability of error $\eps\in[0,1/2]$. For logistic and hinge-loss, the set of minimizers in \eqref{eq:gen_opt} is bounded (as required by Theorem \ref{thm:main}) iff $\delta>\delta^\star_\eps$. See Remark \ref{rem:threshold} and \cite{candes2018phase}.}
          \label{fig:thresholdfig}
 \end{figure}
\begin{remark}[Why $\delta>1$] The theorem assumes that $\delta>1$ or equivalently $m>n$. Here, we show that this condition is \emph{necessary} for the equations \eqref{eq:eq_main} to have a bounded solution, or equivalently for the minimizer $\xh_\ell$ in \eqref{eq:gen_opt} to be bounded.
To see this, take squares in both sides of \eqref{eq:lambda_main} and divide by \eqref{eq:alpha_main}, to find that 
$$
\delta = \frac{\Exp\left[\,\left(\envdx{\ell}{\ourx}{\la}\right)^2\,\right]}{\left(\E\left[ G\cdot \envdx{\ell}{\ourx}{\la} \right]\right)^2} \geq 1.
$$
The inequality follows by applying Cauchy-Schwarz and using the fact that $\E[G^2]=1$.
\end{remark}
\begin{remark}[Solving the equations]\label{rem:FP}
 Evaluating the performance of $\xh_\ell$ requires solving the system of non-linear equations in \eqref{eq:gen_opt}. We empirically observe (see also \cite{Master} for similar observation) that if a solution exists, then it can be efficiently found by the following fixed-point iteration method. Let $\vb := [\mu,\alpha,\la]^T$ and $\Fc:\R^3\rightarrow\R^3$ be such that \eqref{eq:gen_opt} is equivalent to $\vb = \Fc(\vb)$. With this notation, initialize $\vb=\vb_0$ and for $k\geq 1$ repeat the iterations $\vb_{k+1} = \Fc(\vb_k)$ until convergence. 
\end{remark}

\subsection{Performance bounds}

In this section, we use the asymptotic prediction of Theorem \ref{thm:main} to derive an upper bound on the correlation value $\corr{\xh_\ell}{\x_0}$ that holds for \emph{all} choices of continuously differentiable loss functions. The following theorem is the key intermediate result in this direction. At a high-level, the proof of the theorem involves a careful algebraic manipulation of the system of equations \eqref{eq:eq_main}, and, leveraging properties of the Moreau envelope function.
%

\begin{thm}[Lower bound on $\sigma_{\ell}$]\label{sec:lem}Let $\ell$ be continuously differentiable and let assumptions and notation of Theorem \ref{thm:main} hold. Then,
\bea \label{eq:lem}
\sig_\ell^2 \, \Ic\big(\sig_\ell\, G + SY)  \geq \frac{1}{\delta}. 
\eea
\end{thm}
\begin{proof} 
First, applying Gaussian integration by parts in \eqref{eq:lambda_main} we find that : $$1=\lambda\, \delta\,\E[\envddx{\ell}{\ourx}{\la}],$$
where $\envddx{\ell}{x}{\la}$ denotes the second derivative of the Moreau-envelope function with respect to the first argument (which exists since $\ell$ is differentiable). Solving for $\lambda$ and substituting in \eqref{eq:alpha_main} gives :
\bea\label{eq:lower}
\alpha^2 = \frac{1}{\delta}\cdot\frac{ \E\left[\left(\envdx{\ell}{\ourx}{\la}\right)^2\right] }{\left(\,\E[\envddx{\ell}{\ourx}{\la}]\,\right)^2}.
\eea
The rest of the proof follows steps similar to \cite[Lem.~3.4]{donoho2016high} and \cite[Rem.~5.3.3]{Master}. 
Call $$Z:=\mu S Y\quad\text{and}\quad W:=\alpha G + Z.$$
Note that for $\alpha\neq 0$, $W$ has a continuously differentiable function at every $w$; in fact,
$$
p_W^{'}(w) = \int \phi_{\alpha}^{'}(u) p_{Z}(w-u)\, \mathrm{d}u,
$$
where $\phi_{\alpha}(u)=\frac{1}{\alpha\sqrt{2\pi}}e^{-\frac{u^2}{2\alpha^2}}$.

With this notation, by Cauchy-Schwartz inequality we have:
\begin{align*}
\E\left[\left(\envdx{\ell}{W}{\la}\right)^2\right]\cdot \Ic(W) 
&\geq\left(\E[\envdx{\ell}{W}{\la}\cdot \ksi_W]\right)^2\\
&=\left(\E[\envddx{\ell}{W}{\la}]\right)^2,
\end{align*}
where the last line follows by integration by parts.
This, when combined with \eqref{eq:lower} shows that: 
\bea\label{eq:lower2}
\alpha^2 \geq \frac{1}{\delta}\, \frac{1}{\Ic(W)}.
\eea
Recall that $W=\alpha G + \mu SY$ and $G\sim\Nn(0,1)$. Now, since (e.g., \cite[Eqn.~2.13]{barron1984monotonic})
$$
\Ic(c\cdot H) = \Ic(H)/c^2,
$$
we have that
$$
\Ic(W) = \Ic\left(\mu \,\Big(\frac{\alpha}{\mu} G + SY \Big)\right) = \frac{1}{\mu^2}\Ic\big(\sig_\ell\, G + SY\big),
$$
where recall that 
$$
\sig_\ell := \alpha\big/\mu.
$$
Substituting this in \eqref{eq:lower2} gives the desired inequality in the statement of the theorem. 
\end{proof}




%
As it is illustrated in Figure \ref{fig:sigma},  $\sig_\ell^2 \, \Ic\big(\sig_\ell\, G + SY)$ is an increasing function of $\sigma_{\ell}$. Therefore the minimum value of $\sigma_{\ell}$ can be derived from the inequality in \eqref{eq:lem}, which corresponds to the optimum $\sigma_{\ell}$ that can be achieved for any continuously differentiable loss function. In the following two remarks, this result is extended to achieve the optimum values for correlation.

\begin{remark}[Upper bound on correlation] \label{rem:nub}
Using \eqref{eq:lem} with $Y = \BSC_{\epsilon}(\sign(S))$ and by numerically evaluating $\Ic(\sigma_{\ell}G+SY)$ based on $\sigma_{\ell}$\,, we can find a numerical lower bound on $\sigma_{\ell}$\,. In view of \eqref{eq:corr_thm} this directly yields an upper bound on $\corr{\xh_\ell}{\x_0}$, i.e., the maximum correlation that can be achieved for any choice of (continuously differentiable) convex loss function. 
In Figures \ref{fig:fig2}-\ref{fig:fig4}, the green lines show the upper bound on correlation for specific values of $\eps$; see Section \ref{sec:sim} for details. 
\end{remark}
\begin{figure}
  \centering
    \includegraphics[width=7.9cm,height=6.6cm]{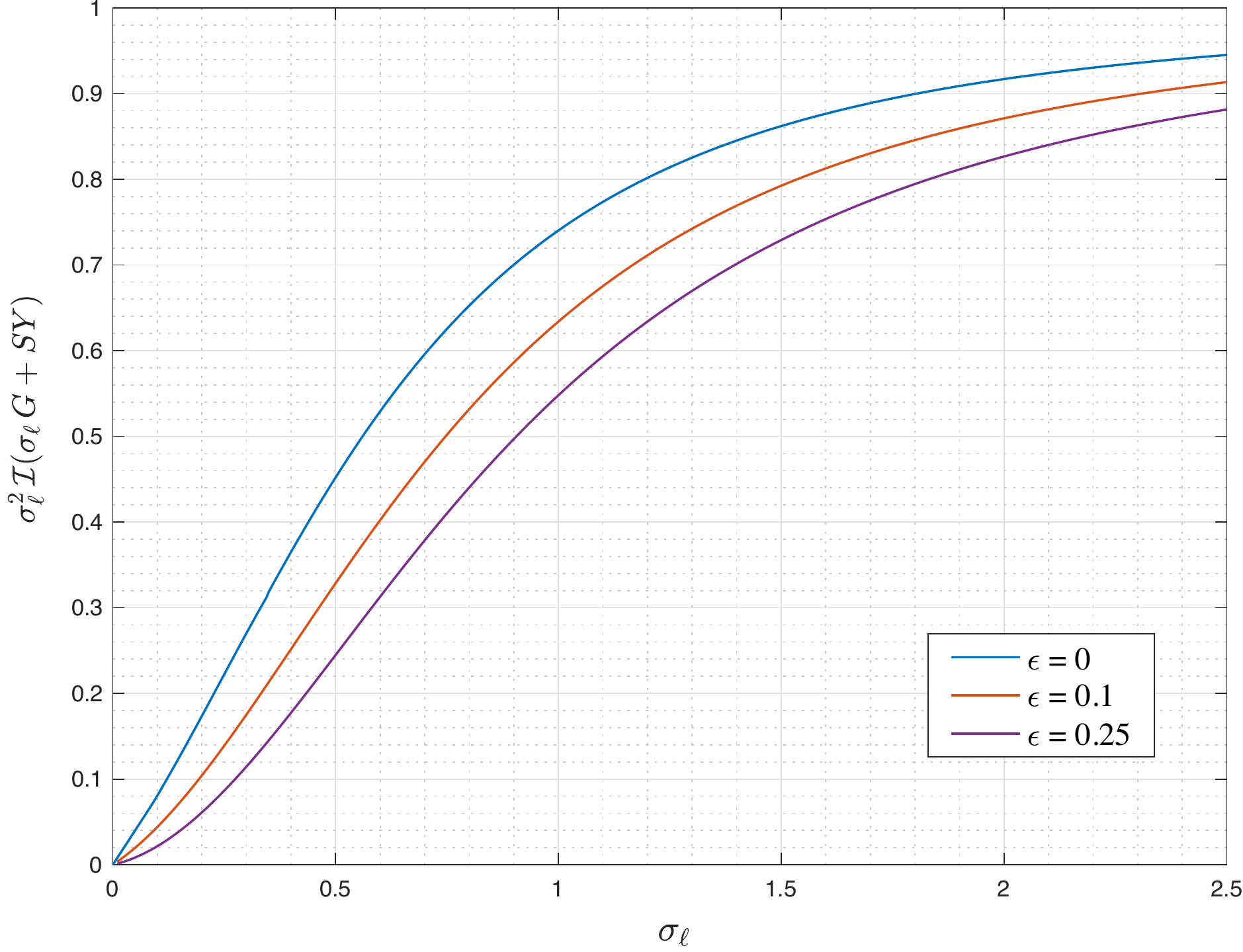}
      \caption{Numerical evaluations of $\sigma_{\ell}^2\Ic(\sigma_{\ell}G+SY)$, as defined in Theorem \ref{sec:lem}, with respect to $\sigma_{\ell}$ for $\epsilon$ = 0, 0.1 and 0.25 (recall \eqref{eq:GSY}). $\sigma_{\ell}^2\Ic(\sigma_{\ell}G+SY) \in [0,1)$ and it is an increasing function of $\sigma_{\ell}$, which implies the existence and uniqueness of minimum value of $\sigma_{\ell}$ for all $\delta>1$.
     }
                  \label{fig:sigma}
\end{figure}



\section{Special cases}\label{sec:cases}

In this section, we apply the general result of Theorem \ref{thm:main} to specific popular choices of the loss function.

\subsection{Least-squares}\label{sec:LS}

By choosing $\ell(t)=(t-1)^2$ in \eqref{eq:gen_opt}, we obtain the standard least squares estimate. To see this, note that since $y_i=\pm 1$, it holds for all $i$ that
$
(y_i\ab_i^T\x-1)^2 = (y_i-\ab_i^T\x)^2.
$

Thus, the estimator $\xh$ is minimizing the sum of squares of the residuals:
\bea\label{eq:LS}
\xh=\arg\min_\x \sum(y_i-\ab_i^T\x)^2.
\eea

For the choice $\ell(t)=(t-1)^2$, it turns out that we can solve the equations in \eqref{eq:eq_main} in closed form. The final result is summarized in the corollary below. 

\begin{cor}[Least-squares]\label{cor:LS}
 Let Assumption \ref{ass:Gaussian} hold and $\delta>1$. Let $\xh$ be as in \eqref{eq:LS}. Then, in the limit of $m,n\rightarrow+\infty$, $m/n\rightarrow\delta$, Equations \eqref{eq:corr_thm} and \eqref{eq:norm_thm} hold with probability one with $\alpha$ and $\mu$ given as follows:
%
\bea
\mu &= (1-2\epsilon)\sqrt{\frac{2}{\pi}}, \label{eq:mu_LS} \\
\alpha^2 &=\left(1-(1-2\epsilon)^2\,\frac{2}{\pi}\right)\frac{1}{\delta-1}.\label{eq:alpha_LS}
\eea
\end{cor}

\begin{proof} In order to get the values of $\alpha$ and $\mu$ as in the statement of the corollary, we show how to simplify Equations \eqref{eq:eq_main} for $\ell(t)=(t-1)^2$. In this case, the proximal operator admits a simple expression:
\begin{equation}
\prox{\ell}{x}{\la} = ({x+2\la})\Big/({1+2\la}).\nn
\end{equation}
Also, $\elld(t)=2(t-1)$.
Substituting these in \eqref{eq:mu_main2} gives the formula for $\mu$ as follows:
\bea\nn
0 &= \E\left[YS(\alpha G + \mu SY - 1)\right] = \mu\, \E[S^2] - \E[YS]\\
&\qquad\qquad\qquad \Longrightarrow
\mu = \sqrt{\frac{2}{\pi}}(1-2\epsilon), \nn
\eea
where we have also used from \eqref{eq:GSY} that $\E[S^2]=1$, $\E[YS]=(1-2\eps)\sqrt{\frac{2}{\pi}}$ and $G$ is independent of $S$.
Also, since $\elldd(t)=2$, direct application of \eqref{eq:lambda_main3} gives
\bea
1 = \lambda\delta\,\frac{2}{1+2\la}\Longrightarrow \la = \frac{1}{2(\delta-1)}\nn.
\eea
Finally, substituting the value of $\lambda$ in \eqref{eq:alpha_main2} we obtain the desired value for $\alpha$ as follows
\begin{align*}
\alpha^2 &= 4 \lambda^2 \delta\,\mathbb{E}\left[(\prox{\ell}{\ourx}{\la}-1)^2\right] \\
&= \frac{4\lambda^2}{(1+2\la)^2}\delta\,\mathbb{E}\left[(\alpha G+\mu S Y -1)^2 \right] \\
&=\frac{4\la^2\delta}{(1+2\la)^2}(\alpha^2+1 -\frac{2}{\pi}(1-2\epsilon)^2)\\
&=\frac{1}{\delta}(\alpha^2+1-\frac{2}{\pi}(1-2\epsilon)^2)\label{eq:alpha}\quad \Longrightarrow \,\eqref{eq:alpha_LS} .
\end{align*}
\end{proof}

\begin{remark}[Least-squares: One-bit vs signed measurements]
On the one hand, Corollary \ref{cor:LS} shows that least-squares for (noisy) one-bit measurements lead to an estimator that satisfies
\begin{equation}\label{eq:norm_LS}
\lim_{n\rightarrow \infty} \Big\|\xh-\frac{\mu}{\|\x_0\|_2}\cdot{\x_0}\Big\|_2^2 = \tau^2\cdot \frac{1}{\delta-1},
\end{equation}
where $\mu$ is as in \eqref{eq:mu_LS} and $\tau^2:=1-(1-2\eps)\frac{2}{\pi}$. On the other hand, it is well-known (e.g., see references in \cite[Sec.~5.1]{Master}) that least-squares for (scaled) linear measurements with additive Gaussian noise (i.e. $y_i= \rho \ab_i^T\x_0 + \sigma z_i$, $z_i\sim\Nn(0,1)$) 
leads to an estimator that satisfies
\bea
\lim_{n\rightarrow \infty} \|\xh-\rho\cdot{\x_0}\|_2^2 = \sigma^2\cdot \frac{1}{\delta-1}.\label{eq:norm_LS_lin}
\eea
Direct comparison of \eqref{eq:norm_LS} to \eqref{eq:norm_LS_lin} suggests that least-squares with one-bit measurements performs the same as if measurements were linear with scaling factor $\rho=\mu/\|\x_0\|_2$ and noise variance $\sigma^2=\tau^2=\alpha^2(\delta-1)$. This worth-mentioning conclusion is not new; it was proved in \cite{Bri,PV15,NIPS}. We elaborate on the relation to this prior work in the following remark.
\end{remark}

\begin{remark}[Prior work]
There is a lot of recent work on the use of least-squares-type estimators for recovering signals from nonlinear measurements of the form $y_i=f(\ab_i^T\x_0)$ with Gaussian vectors $\ab_i$. The original work that suggests least-squares as a reasonable estimator in this setting is due to Brillinger \cite{Bri}. In his 1982 paper, Brillinger studied the problem in the classical statistics regime (aka $n$ is fixed not scaling with $m\rightarrow+\infty$) and he proved for the least-squares solution satisfies
$$
\lim_{m\rightarrow+\infty} \frac{1}{m}\|\xh-\frac{\mu}{\|\x_0\|_2}\cdot{\x_0}\|_2^2 = \tau^2,
$$
where
\begin{align}
\mu &= \E[S f(S)],\quad\quad\qquad S\sim\Nn(0,1),\nn\\
\tau^2 &= \E[(f(S)-\mu S)^2].\label{eq:Bri}
\end{align}
and the expectations are with respect to $S$ and possible randomness of $f$. Evaluating \eqref{eq:Bri} for $f(S)=\BSCe(\sign(S))$ leads to the same values for $\mu$ and $\tau^2$ in \eqref{eq:norm_LS}. In other works, \eqref{eq:norm_LS} for $\delta\rightarrow+\infty$ indeed recovers Brillinger's result. The extension of Brillinger's original work to the high-dimensional setting (both $m,n$ large) was first studied by Plan and Vershynin \cite{PV15}, who derived (non-sharp) non-asymptotic upper bounds on the performance of constrained least-squares (such as the Lasso). Shortly after, \cite{NIPS} extended this result to \emph{sharp} asymtpotic predictions and to regularized least-squares. In particular, Corollary \ref{cor:LS} is a special case of the main theorem in \cite{NIPS}. Several other interesting extensions of the result by Plan and Vershynin have recently appeared in the literature, e.g., \cite{genzel2017high,goldstein2018structured,genzel2017recovering,thrampoulidis2018generalized}. However, \cite{NIPS} is the only one to give results that are sharp in the flavor of this paper. Our work, extends the result of \cite{NIPS} to general loss functions beyond least-squares. The techniques of \cite{NIPS} that have guided the use of the CGMT in our context have also been recently applied in \cite{PhaseLamp} in the context of phase-retrieval.
%
\end{remark}


\begin{figure}
  \centering
    \includegraphics[width=8cm,height=6.6cm]{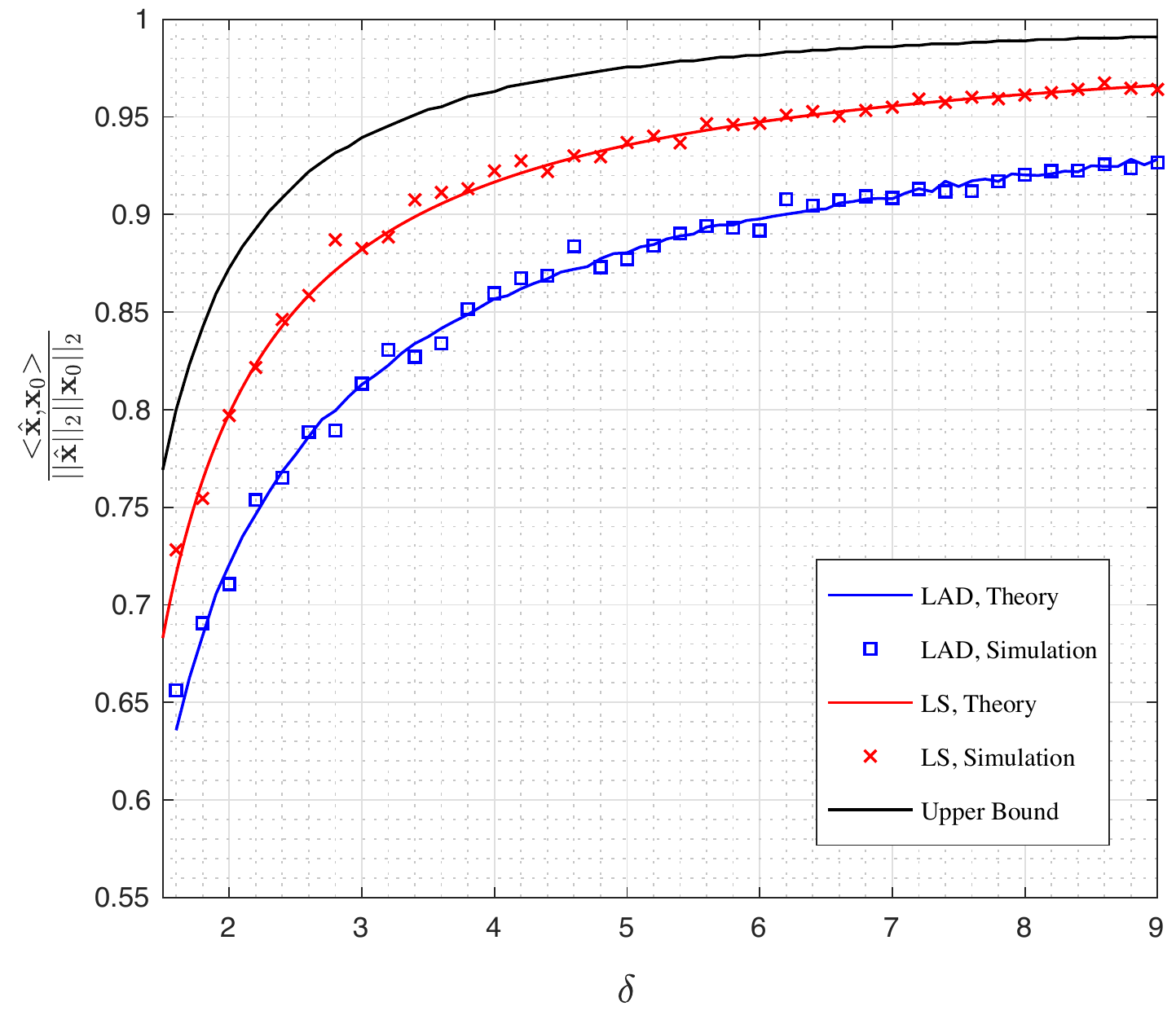}
      \caption{Comparisons between theoretical and simulated results for the least-squares (LS) and least-absolute deviations (LAD) estimators along with the upper bound, as a function of $\delta$, for noiseless measurements ($\epsilon=0$). The LS estimator significantly outperforms the LAD for all values of $\delta$.}
      \label{fig:fig2}
\end{figure}

\subsection{Least-absolute deviations} \label{sec:LAD}

By choosing $\ell(t)=|t-1|$ in \eqref{eq:gen_opt}, we obtain a least-absolute deviations estimate. Again, since $y_i=\pm 1$, it holds for all $i$ that
$
|y_i\ab_i^T\x-1| = |y_i-\ab_i^T\x|.
$
Thus, this choice of loss function leads to residuals:
\bea\label{eq:LAD}
\xh=\arg\min_\x \sum |y_i-\ab_i^T\x|.
\eea

As in Section \ref{sec:LS}, for $\ell(t) = |t-1|$ the proximal operator admits a simple expression, as follows:
\bea
\prox{\ell}{x}{\lambda} = 1+ \soft{x-1}{\la}
\eea
where 
$$
\soft{x}{\la} =   \begin{cases}
      x-\lambda, & \text{if}\ x>\lambda, \\
      x+\lambda, & \text{if}\ x<-\lambda, \\
      0,&\text{otherwise}.
    \end{cases}
$$
is the standard soft-thresholding function.


\subsection{Hinge-loss}
We obtain the hinge-loss estimator in by setting $\ell(t) = \max(1-t,0)$ in \eqref{eq:gen_opt}. Similar to Section \ref{sec:LAD}, the  proximal operator of the hinge-loss can be expressed in terms of the soft-thresholding function as follows:
\begin{align*}
\prox{\ell}{x}{\la} = 1 +\soft{x+\frac{\la}{2}-1}{\frac{\la}{2}}.
\end{align*}
As already mentioned in Remark \ref{rem:threshold}, the set of minimizers of the hinge-loss is bounded (required by Theorem \ref{thm:main}) only for $\delta>\delta^\star_\eps$ where $\delta^\star_\eps$ is the value of the threshold in \eqref{eq:threshold}. Our numerical simulations in Figures \ref{fig:fig3} and \ref{fig:fig4} suggest that hinge-loss is robust to measurement corruptions, as for moderate to large values of $\delta$ it outperforms the LS and the LAD estimators. Theorem \ref{thm:main} opens the way to analytically confirm such conclusions, which is an interesting future direction.

\subsection{Numerical simulations}\label{sec:sim}

\begin{figure}
  \centering
    \includegraphics[width=7.9cm,height=6.6cm]{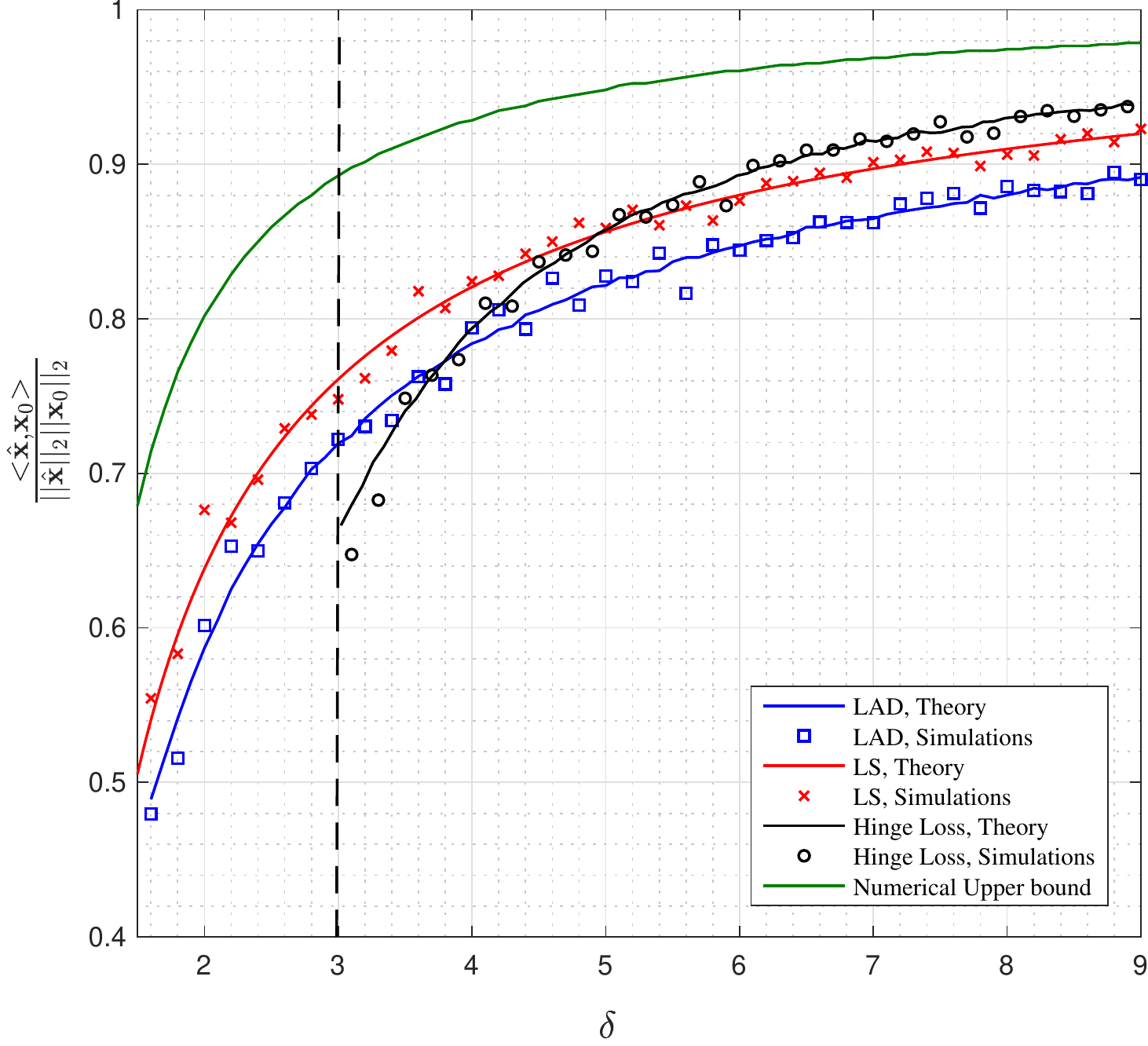}
      \caption{Comparison between theoretical and simulated results for LAD, LS and Hinge-Loss estimators along with numerical upper bound, as a function of $\delta$ for probability of error $\epsilon=0.1$. The dashed line represents the value of the threshold $\delta^*_\eps$ for $\epsilon=0.1$ (see Figure \ref{fig:thresholdfig}). For small values of $\delta$ LS outperforms the other two estimators, but the hinge-loss becomes better as $\delta$ in increases.}
            \label{fig:fig3}
            \end{figure}
            
\begin{figure}
  \centering
    \includegraphics[width=7.9cm,height=6.6cm]{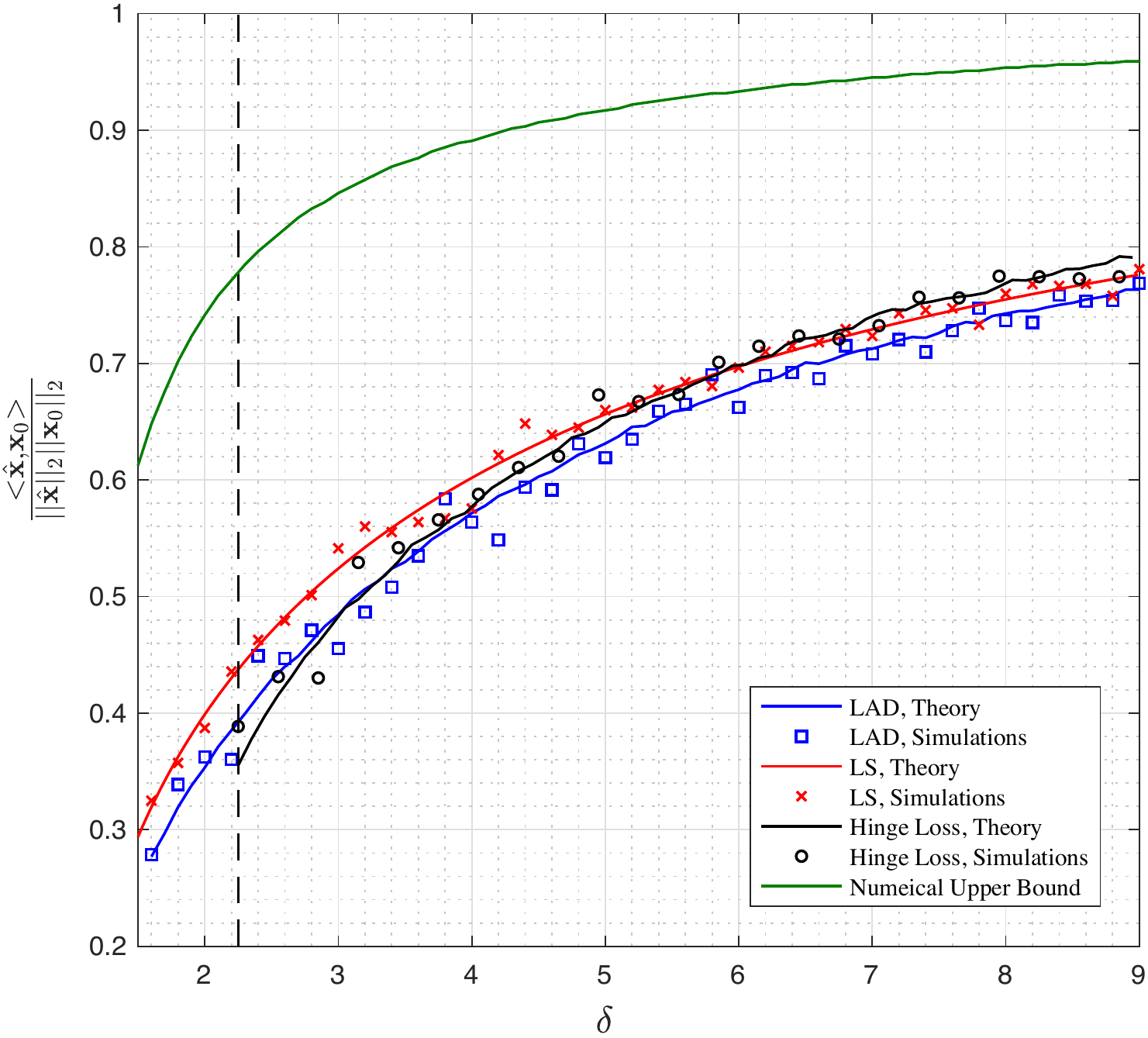}
      \caption{
      Comparison between theoretical and simulated results for LAD, LS and Hinge-Loss estimators along with numerical upper bound, as a function of $\delta$ for probability of error $\epsilon=0.25$. As in Figure \ref{fig:fig3}, the dashed line represents the value of the threshold $\delta^*_\eps$ for $\epsilon=0.25$.}
                  \label{fig:fig4}
\end{figure}

We present numerical simulations that validate the predictions of Theorems \ref{thm:main} and \ref{sec:lem} . For the numerical experiments, we generate random measurements according to \eqref{eq:gen_model} and Assumption \ref{ass:Gaussian}. Without loss of generality (due to rotational invariance of the Gaussian measure) we set $\x_0=[1,0,...,0]^T$. We then obtain estimates $\xh_\ell$ of $\x_0$ by numerically solving \eqref{eq:gen_opt} and measuring performance by the correlation value $\corr{\hat{\x}_\ell}{\x_0}$. Throughout the experiments, we set $n=128$ and the recorded values of correlation in Figures \ref{fig:fig2}--\ref{fig:fig4} are averages over $25$ independent experiments. The theoretical curves for the correlation are computed based on Theorem \ref{thm:main}. We solve the system of equations in \eqref{eq:eq_main} by the fixed-point iteration method described in Remark \ref{rem:FP}. The expectations involved in \eqref{eq:eq_main} are evaluated with Monte-Carlo estimation using $10^5$ independent samples. Numerical upper bounds in Figures \ref{fig:fig2}-\ref{fig:fig4} are derived according to Remark \ref{rem:nub}.
%
%
%
%
%

Comparisons between theoretical and simulated values for LAD and LS estimators along with the upper bound on correlation are illustrated in Figure \ref{fig:fig2} for the noiseless case. Note that for $\eps=0$, the hinge-loss has an unbounded set of minimizers for all values of $\delta$ (thus, Theorem \ref{thm:main} is not applicable). In Figure \ref{fig:fig3},  the probability of error $\epsilon$ is increased to $0.1$. Note that in this setting hinge-loss estimator exists for $\delta> \delta^*_{0.1}\approx3$ and that it outperforms LAD and LS estimators for large values of $\delta$. \\
In Figure \ref{fig:fig4} we present similar results for $\epsilon=0.25$. 
As it is evident from Figures \ref{fig:fig3} and \ref{fig:fig4}, the best estimator is varying based on the value of $\delta$ and $\eps$. 
This further emphasizes the impact of studying the accuracy of estimators while we do not restrict ourselves to a specific loss function.

\section{Conclusion}\label{sec:conc}

This paper derives \emph{sharp} asymptotic performance guarantees for a wide class of convex optimization based estimators for recovering a signal from corrupted one-bit measurements. Our general result includes as a special case the least-squares estimator that was previously studied in \cite{NIPS}. Beyond that, it applies to other popular estimators such as the LAD, Hinge-loss, logistic loss, etc. One natural and interesting research direction  is finding the optimal loss function $\ell(\cdot)$ in \eqref{eq:gen_opt}. In view of Theorem \ref{thm:main}, this boils down to finding $\ell(\cdot)$ that minimizes the ratio $\alpha/\mu$ of the parameters $\alpha$ and $\mu$ that solve the system of equations in \eqref{eq:eq_main}. For this purpose, it might also be important to derive necessary and sufficient conditions that guarantee \eqref{eq:eq_main} has a unique solution. 

\bibliographystyle{plain}
\bibliography{compbib}

\appendix
\section{Analysis}\label{sec:proof}

In this section we provide a proof sketch of Theorem \ref{thm:main}. The main technical tool that facilitates our analysis is the convex Gaussian min-max theorem (CGMT), which is an extension of Gordon's Gaussian min-max inequality (GMT). We introduce the necessary background on the CGMT in \ref{sec:CGMT}.

\subsection{Technical tool: CGMT}\label{sec:CGMT}

\subsubsection{Gordon's Min-Max Theorem (GMT)}
The Gordon's Gaussian comparison inequality \cite{Gor88} compares the min-max value of two doubly indexed Gaussian processes based on how their autocorrelation functions compare. The inequality is quite general (see \cite{Gor88}), but for our purposes we only need its application to the following two Gaussian processes:
\begin{subequations}
\begin{align}
X_{\w,\ub} &:= \ub^T \G \w + \psi(\w,\ub),\\
Y_{\w,\ub} &:= \norm{\w}_2 \g^T \ub + \norm{\ub}_2 \h^T \w + \psi(\w,\ub),
\end{align}
\end{subequations}
where: $\G\in\mathbb{R}^{m\times n}$, $\g \in \mathbb{R}^m$, $\h\in\mathbb{R}^n$, they all have entries iid Gaussian; the sets $\mathcal{S}_{\w}\subset\R^n$ and $\mathcal{S}_{\ub}\subset\R^m$ are compact; and, $\psi: \mathbb{R}^n\times \mathbb{R}^m \to \mathbb{R}$. For these two processes, define the following (random) min-max optimization programs, which (following \cite{Master}) we refer to as the \emph{primary optimization} (PO) problem and the \emph{auxiliary optimization} (AO) -- for purposes that will soon become clear. 
\begin{subequations}
\begin{align}\label{eq:PO_loc}
\Phi(\G)&=\min\limits_{\w \in \mathcal{S}_{\w}} \max\limits_{\ub\in\mathcal{S}_{\ub}} X_{\w,\ub},\\
\label{eq:AO_loc}
\phi(\g,\h)&=\min\limits_{\w \in \mathcal{S}_{\w}} \max\limits_{\ub\in\mathcal{S}_{\ub}} Y_{\w,\ub}.
\end{align}
\end{subequations}

According to Gordon's comparison inequality, for any $c\in\R$, it holds:
\begin{equation}\label{eq:gmt}
\mathbb{P}\left( \Phi(\G) < c\right) \leq 2 \mathbb{P}\left(  \phi(\g,\h) < c \right).
\end{equation}
In other words, a high-probability lower bound on the AO is a high-probability lower bound on the PO. The premise is that it is often much simpler to lower bound the AO rather than the PO. To be precise, \eqref{eq:gmt} is a slight reformulation of Gordon's original result proved in \cite{COLT} (see therein for details).
\subsubsection{Convex Gaussian Min-Max Theorem (CGMT)}
The proof of Theorem \ref{thm:main} builds on the CGMT \cite{COLT}. 
For ease of reference we summarize here the essential ideas of the framework following the presentation in \cite{Master}; please see \cite[Section~6]{Master} for the formal statement of the theorem and further details.
The CGMT is an extension of the GMT and it asserts that the AO in \eqref{eq:AO_loc} can be used to tightly infer properties of the original (PO) in \eqref{eq:PO_loc}, including the optimal cost and the optimal solution.
According to the CGMT \cite[Theorem 6.1]{Master}, if the sets $\mathcal{S}_{\w}$ and $\mathcal{S}_{\ub}$ are convex and $\psi$ is continuous \emph{convex-concave} on $\mathcal{S}_{\w}\times \mathcal{S}_{\ub}$, then, for any $\nu \in \mathbb{R}$ and $t>0$, it holds
\begin{equation}\label{eq:cgmt}
\mathbb{P}\left( \abs{\Phi(\G)-\nu} > t\right) \leq 2 \mathbb{P}\left(  \abs{\phi(\g,\h)-\nu} > t \right).
\end{equation}
In words, concentration of the optimal cost of the AO problem around $\mu$ implies concentration of the optimal cost of the corresponding PO problem around the same value $\mu$.  Moreover, starting from \eqref{eq:cgmt} and under strict convexity conditions, the CGMT shows that concentration of the optimal solution of the AO problem implies concentration of the optimal solution of the PO to the same value. For example, if minimizers of \eqref{eq:AO_loc} satisfy $\norm{\w^\ast(\g,\h)}_2 \to \zeta^\ast$ for some $\zeta^\ast>0$, then, the same holds true for the minimizers of \eqref{eq:PO_loc}: $\norm{\w^\ast(\G)}_2 \to \zeta^\ast$ \cite[Theorem 6.1(iii)]{Master}. Thus, one can analyze the AO to infer corresponding properties of the PO, the premise being of course that the former is simpler to handle than the latter.

\subsection{Applying the CGMT to \eqref{eq:opt_reg}} \label{sec:mainproof}


In this section, we show how to apply the CGMT to \eqref{eq:opt_reg}. For convenience, we drop the subscript $\ell$ from $\xh_\ell$ and simply write
\begin{equation} \label{eq:opt}
\xh = \arg\min_{\x} \frac{1}{m} \sum_{i=1}^{m} \ell(y_i \ab_i^T \x) + r\left\lVert\x\right\rVert_2^2, 
\end{equation}
where the measurements $y_i,~i\in[m]$ follow \eqref{eq:gen_model}. By rotational invariance of the Gaussian distribution of the measurement vectors $\ab_i,~i\in[m]$, we assume without loss of generality that $\x_0 = [1,0,...,0]^T$. Denoting $y_i\ab_i^T\x$ by $u_i$, \eqref{eq:opt} is equivalent to the following min-max optimization: 
\begin{dmath}\label{eq:mmbn}
\min_{\ub,\x} \max_{\pmb{\beta}} \frac{1}{m} \sum_{i=1}^{m}\ell(u_i) + \frac{1}{m}\sum_{i=1}^{m}\beta_i u_i \\
- \frac{1}{m}\sum_{i=1}^{m}\beta_i y_i \ab_i^T \x + r\left\lVert\x\right\rVert_2^2.
\end{dmath}
Now, let us define 
$$\ab_i=[s_i;\tilde{\ab}_i],~i\in[m]\quad\text{ and }\quad \x=[x_1;\tilde{\x}],$$ such that $s_i$ and $x_1$ are the first entries of $\ab_i$ and $\x$, respectively. Note that in this new notation \eqref{eq:gen_label} becomes:
\bea\label{eq:y_pf}
y_i = \begin{cases} 1 &, \text{w.p.}~~ f(s), \\ -1 &, \text{w.p.}~~1-f(s),  \end{cases}
\eea
and 
\bea\label{eq:corr_pf}
\corr{\x}{\x_0} = \frac{\xh_1}{\sqrt{\xh_1^2+\|\widetilde{\xh}\|_2^2}},
\eea
where we denote $\xh=[\xh_1;{\widetilde{\xh}}]$.
Also, \eqref{eq:mmbn} is written as 
\begin{dmath*}
\min_{\ub,\x}\max_{\pmb{\beta}}  \frac{1}{m}\sum_{i=1}^{m} \ell(u_i) + \frac{1}{m}\sum_{i=1}^{m} \beta_i u_i+ \frac{1}{m}\sum_{i=1}^{m} \beta_i y_i \tilde{\ab}_i^T \tilde{\x} - \frac{1}{m}\sum_{i=1}^{m} \beta_i y_i s_i x_1+r x_1^2 + r\left\lVert\tilde{\x}\right\rVert_2^2   
\end{dmath*}
or, in matrix form:
\begin{dmath}\label{eq:normopt_PO}
\min_{\ub,\x}\max_{\pmb{\beta}}~ \frac{1}{m}\betab^T\mathbf{D}_y\tilde{\A}\tilde{\x} + \frac{1}{m}x_1\pmb{\beta}^T\mathbf{D}_y\s+ \frac{1}{m}\pmb{\beta}^T\ub + r x_1^2 + r\left\lVert\tilde{\x}\right\rVert_2^2+\frac{1}{m}\sum_{i=1}^m \ell (u_i).
\end{dmath}
where $\mathbf{D}_\y := {\rm{diag}}(y_1,y_2,...,y_m)$ is a diagonal matrix with $y_1,y_2,...y_m$ on the diagonal, $\s=[s_1,\ldots,s_m]^T$ and $\tilde{\A}$ is an $m\times (n-1)$ matrix with rows $\tilde{\ab}_i^T,~i\in[m]$.

In \eqref{eq:normopt_PO} we recognize that the first term has the bilinear form required by the GMT in \eqref{eq:PO_loc}. The rest of the terms form the function $\psi$ in \eqref{eq:PO_loc}: they are independent of $\tilde\A$ and convex-concave as desired by the CGMT. Therefore, we have expressed \eqref{eq:opt} in the desired form of a PO and for the rest of the proof we will analyze the probabilistically equivalent  AO problem. In view of \eqref{eq:AO_loc}, this is given as follows,
\begin{dmath}\label{eq:normopt}
\min_{\ub,\x}\max_{\pmb{\beta}} \frac{1}{m} \left\lVert\tilde{\x}\right\rVert_2 \mathbf{g}^T \mathbf{D}_y \pmb{\beta} + \frac{1}{m} \left\lVert \mathbf{D}_y\pmb{\beta}\right\rVert_2\mathbf{h}^T\tilde{\x} - \frac{1}{m}x_1\pmb{\beta}^T\mathbf{D}_y\s+ \frac{1}{m}\pmb{\beta}^T\ub +r x_1^2 + r\left\lVert\tilde{\x}\right\rVert_2^2+ \frac{1}{m}\sum_{i=1}^m \ell (u_i) ,
\end{dmath}
where as in \eqref{eq:AO_loc} $\mathbf{g}\sim\mathcal{N}(0,I_m)$ and $\mathbf{h}\sim\mathcal{N}(0,I_{n-1})$.

\subsection{Analysis of the Auxiliary Optimization}
Here, we show how to analyze the AO in \eqref{eq:normopt} \footnote{There are several technical details in the proof that we omit from this proof sketch. This includes: boundedness of the constraint sets in \eqref{eq:normopt}; changing the order of min-max when optimizing over the direction of $\tilde{\x}$ in \eqref{eq:alpha_pf}; and, uniform convergence in going from \eqref{eq:binfty} to \eqref{eq:det}.}. The steps are similar and follow the recipe prescribed in \cite{COLT,Master}. To begin with, note that $y_i\in{\pm1}$, therefore $\mathbf{D}_\y\mathbf{g} \sim\mathcal{N}(0,I_m)$ and $\left\lVert\mathbf{D}_\y\pmb{\beta}\right\rVert_2 = \left\lVert\pmb{\beta}\right\rVert_2$. Let us denote 
$$\alpha:=\left\lVert\tilde{\x}\right\rVert_2\quad\text{ and }\quad\mu:=x_1.$$ We can simplify \eqref{eq:normopt} by optimizing over the direction of $\tilde{\x}$. This leads to the following:
\begin{dmath}\label{eq:alpha_pf}
\min_{\alpha\ge0,\mu,\ub}~\max_{\pmb{\beta}} ~ \frac{1}{m}\alpha\mathbf{g}^T\pmb{\beta} - \frac{\alpha}{m}\left\lVert\pmb{\beta}\right\rVert_2\left\lVert\mathbf{h}\right\rVert_2 - \frac{1}{m}\mu \mathbf{s}^T\mathbf{D}_{\y} \pmb{\beta} + \frac{1}{m}\pmb{\beta}^T \mathbf{u} +r\mu^2 + r\alpha^2 + \frac{1}{m} \sum_{i=1}^m\ell(u_i). 
\end{dmath}
Next, let $\gamma := \frac{\left\lVert\pmb{\beta}\right\rVert_2}{\sqrt{m}}$ and optimize over the direction of $\betab$ to yield
\begin{dmath}\label{eq:gamma}
\min_{\alpha\ge0,\ub,\mu}~\max_{\gamma\ge0}~\frac{\gamma}{\sqrt{m}}\left\lVert \alpha\mathbf{g}-\mu\mathbf{D}_\y\mathbf{s}+\ub\right\rVert_2 - \frac{\alpha}{\sqrt{m}}\gamma\left\lVert\mathbf{h}\right\rVert_2 +r\mu^2 + r\alpha^2 + \frac{1}{m}\sum_{i=1}^m\ell(u_i).
\end{dmath}
To continue, we utilize the fact that for all $x\in\R$, 
$\min_{\tau>0}\frac{\tau}{2} + \frac{x^2}{2\tau m} = \frac{x}{\sqrt{m}}$.
With this trick, the optimization over $\ub$ becomes separable over its coordinates and \eqref{eq:gamma} can be rewritten as
\begin{dmath*}
\min_{\alpha\ge0,\tau>0,\ub,\mu}\max_{\gamma\ge0}\frac{\gamma\tau}{2} + \frac{\gamma}{2\tau m}{\left\lVert  \alpha\mathbf{g}-\mu\mathbf{D}_\y\mathbf{s}+\ub \right\rVert}_2 ^2+r\mu^2 + r\alpha^2 + \frac{1}{m}\sum_{i=1}^m\ell(u_i) - \frac{\alpha}{\sqrt{m}}\gamma\left\lVert\mathbf{h}\right\rVert_2.
\end{dmath*}
Equivalently, we express this in the following convenient form:
\begin{dmath} \label{eq:binfty}
\min_{\mu,\alpha\ge0,\tau>0}~\max_{\gamma\ge0}~\frac{\gamma \tau}{2}- \frac{\alpha}{\sqrt{m}}\gamma\left\lVert\mathbf{h}\right\rVert_2 +r\mu^2 + r\alpha^2 + \frac{1}{m}\sum_{i=1}^m \env{\ell}{\alpha G+\mu YS}{\frac{\tau}{\gamma}},
\end{dmath}
where recall the definition of the Moreau envelope:
$$
\env{\ell}{\alpha g_i+\mu y_is_i}{\frac{\tau}{\gamma}} =  \min_{u_i}\frac{\gamma}{2\tau}(\alpha g_i + \mu y_i s_i-u_i)^2 + \ell(u_i).
$$
As to now, we have reduced the AO into a random min-max optimization over only four scalar variables in \eqref{eq:binfty}. For fixed $\mu,\alpha,\tau,\gamma$, direct application of the weak law of large numbers, shows that the objective function of \eqref{eq:binfty} converges in probability to the following as  $m,n\rightarrow\infty$ and $\frac{m}{n}=\delta$:
$$
\gamma\frac{\tau}{2}-\frac{\alpha\gamma}{\sqrt{\delta}}+r\mu^2 + r\alpha^2 + \mathbb{E}\left[\env{\ell}{\alpha G+\mu YS}{\frac{\tau}{\gamma}} \right],
$$
where $G,S\sim\mathcal{N}(0,1)$ and $Y\sim 2\mathrm{Bernoulli}(f(S))-1$ (in view of \eqref{eq:y_pf}). 
Based on that, it can be shown (see \cite{Master,PhaseLamp} for similar arguments; details are deferred to the long version of the paper) that the random optimizers $\alpha_n$ and $\mu_n$ of \eqref{eq:binfty} converge to the deterministic optimizers $\alpha$ and $\mu$ of the following (deterministic) optimization problem (whenever these are bounded as the statement of the theorem requires):
\begin{dmath}\label{eq:det}
\min_{\alpha\ge0,\mu,\tau>0}\max_{\gamma\ge0} \gamma\frac{\tau}{2}-\frac{\alpha\gamma}{\sqrt{\delta}}+r\mu^2 + r\alpha^2 + \mathbb{E}\left[\env{\ell}{\alpha G+\mu YS}{\frac{\tau}{\gamma}} \right].
\end{dmath}
At this point, recall that $\alpha$ represents the norm of $\tilde\x$ and $\mu$ the value of $x_1$. Thus, in view of (i) \eqref{eq:corr_pf}, (ii) the equivalence between the PO and the AO, and, (iii) our derivations thus far we have that 
\bea
\lim_{n\rightarrow+\infty}\corr{\xh}{\x_0} = \frac{\mu}{\sqrt{\mu^2+\alpha^2}},\nn
\eea
where $\mu$ and $\alpha$ are the minimizers in \eqref{eq:det}. The three equations in \eqref{eq:eq_main} are derived by the first-order optimality conditions of the optimization in \eqref{eq:det}. We show this next.

\subsection{First-order optimality conditions}
By direct differentiation, the first order optimality conditions of of the min-max optimization in \eqref{eq:det} are as follows:
\begin {equation}\label{eq:1}
2r\mu+\mathbb{E}\left[YS\envdx{\ell}{\ourx}{\frac{\tau}{\gamma}}\right] = 0,
\end{equation}
\begin{equation}\label{eq:2}
2 r \alpha +\mathbb{E}\left[G \envdx{\ell}{\ourx}{\frac{\tau}{\gamma}} \right]=\frac{\gamma}{\sqrt{\delta}},
\end{equation}
\begin{equation}\label{eq:3}
\frac{\gamma}{2} + \frac{1}{\gamma}\mathbb{E}\left[\envdla{\ell}{\ourx}{\frac{\tau}{\gamma}}\right]=0,
\end{equation}
\begin{equation}\label{eq:4}
-\frac{\alpha}{\sqrt{\delta}}-\frac{\tau}{\gamma^2}\mathbb{E}\left[\envdla{\ell}{\ourx}{\frac{\tau}{\gamma}}\right]+\frac{\tau}{2} = 0.
\end{equation}
Next, we show how these equations simplify to the following system of equation: 
\begin{subequations}\label{eq:reg_main}
\bea
 \Exp\left[Y\, S \cdot\envdx{\ell}{\ourx}{\la}  \right]&=-2r\mu , \label{eq:mureg_main}\\
 {\la^2}\,{\delta}\,\Exp\left[\,\left(\envdx{\ell}{\ourx}{\la}\right)^2\,\right]&=\alpha^2 ,
\label{eq:alphareg_main}\\
\lambda\,\delta\,\E\left[ G\cdot \envdx{\ell}{\ourx}{\la}  \right]&=\alpha(1-2r\la\delta) .
\label{eq:lambdareg_main}
\eea
\end{subequations}
Denote $\lambda := \frac{\tau}{\gamma}$. 
First, \eqref{eq:mureg_main} is followed by equation \eqref{eq:1}.
Second, substituting $\gamma$ from \eqref{eq:3} in \eqref{eq:4} yields $\tau=\frac{\alpha}{\sqrt{\delta}}$ or $\gamma=\frac{\alpha}{\lambda \sqrt{\delta}}$, which together with \eqref{eq:2} leads to \eqref{eq:lambdareg_main}. Finally, \eqref{eq:alphareg_main} can be obtained by substituting $\gamma = \frac{\alpha}{\lambda\sqrt{\delta}}$ in \eqref{eq:3} and using the relation:
\begin{equation*}
\envdla{\ell}{\ourx}{\lambda} = -\frac{1}{2}(\envdx{\ell}{\ourx}{\lambda})^2.
\end{equation*}
As already mentioned in Section \ref{sec:form}, we focus on the non-regularized loss functions in this paper. Setting $r=0$ in \eqref{eq:reg_main} will give the desired system of equations in \eqref{eq:eq_main}

\end{document}